\documentclass[11pt,notitlepage,twoside,a4paper]{amsart}
 \usepackage{amsfonts}

\usepackage{amsmath,amssymb,enumerate}

\usepackage{epsfig,fancyhdr,color}

\usepackage{amssymb}
\usepackage{amsmath,amsthm}
\usepackage{latexsym}
\usepackage{amscd}
\usepackage{psfrag}
\usepackage{graphicx}
\usepackage[latin1]{inputenc}
\usepackage[all]{xy}
\usepackage[mathcal]{eucal}

\definecolor{NoteColor}{rgb}{1,0,0}

\renewcommand{\textsc}{\textcolor{red}}

\newtheorem{theorem}{\rm\bf Theorem}[section]
\newtheorem{proposition}[theorem]{\rm\bf Proposition}

\newtheorem{corollary}[theorem]{\rm\bf Corollary}
\newtheorem*{theorem 1}{\rm\bf Proposition 1}
\newtheorem*{theorem 2}{\rm\bf Proposition 2}

\theoremstyle{definition}
\newtheorem{definition}[theorem]{\rm\bf Definition}

\theoremstyle{remark}
\newtheorem{remark}[theorem]{\rm\bf Remark}

\newtheorem{example}[theorem]{\rm\bf Example}

\newcommand{\del}{\partial}

\def\interieur#1{\mathord{\mathop{\kern 0pt #1}\limits^\circ}}

\title[Funk and Hilbert geometry]{The Funk and Hilbert geometries for spaces of constant curvature}

\author{Athanase Papadopoulos}
\address{Athanase Papadopoulos,  Institut de Recherche Math\'ematique Avanc\'ee,
Universit{\'e} de Strasbourg and CNRS,
7 rue Ren\'e Descartes,
 67084 Strasbourg Cedex, France.} 
 \email{athanase.papadopoulos@math.unistra.fr}

\author{Sumio Yamada}
\address{Sumio Yamada, Mathematical Institute, Tohoku University, Sendai, 980-8578, Japan}
\email{yamada@math.tohoku.ac.jp}

\date{\today}

\begin{document}

\begin{abstract}

 The goal of this paper is to introduce and study analogues of the 
Euclidean Funk and Hilbert metrics on open convex subsets $\Omega$ of hyperbolic or spherical spaces. At least at a formal level, there are striking similarities among the three cases: Euclidean, spherical and hyperbolic. We start by defining non-Euclidean analogues of the Euclidean Funk weak metric and we give three distinct representations of it in each of the non-Euclidean cases, which parallel the known situation for the Euclidean case. As a consequence, all of these metrics are shown to be  Finslerian, and  the associated norms of the Finsler metrics are described. The theory is developed by using a set of classical trigonometric identities on the sphere $S^n$ and the hyperbolic space $\mathbb{H}^n$ and the definition of a cross ratio on the non-Euclidean spaces of constant curvature.  This in turn leads to the concept of projectivity invariance in these spaces. We then study the geodesics of the Funk and Hilbert metrics. In the case of Euclidean (respectively spherical, hyperbolic) geometry, the Euclidean (respectively spherical, hyperbolic) geodesics are Funk and Hilbert geodesics. Natural projection maps that exist between the spaces $\mathbb{R}^n$, $\mathbb{H}^n$ and the upper hemisphere demonstrate that the theories of Hilbert geometry of convex sets in the three spaces of constant curvature are all equivalent. The same cannot be said   about the Funk geometries.

\bigskip
\noindent AMS Mathematics Subject Classification:    Primary  58B20 ; Secondary 53C60, 51A05,  51A05, 53A35.

\bigskip 

\noindent Keywords: Hilbert metric, Funk metric, constant curvature, non-Euclidean trigonometry,  cross ratio, non-Euclidean geometry, projective geometry, Minkowski space, generalized Beltrami-Klein model. 
 \end{abstract}
\maketitle

\thanks{Second author supported in part by JSPS Grant-in-aid for Scientific Research No.24340009}

\section{Introduction}

Given a bounded open convex subset $\Omega$ of a Euclidean space, D. Hilbert  proposed in \cite{H} (1895) a natural metric $H_\Omega(x, y)= H(x,y)$ on $\Omega$, now called the \emph{Hilbert metric}. It is defined for $x\not= y$ in $\Omega$ as the logarithm of the cross ratio of the quadruple $(x, y, b(x, y), b(y,x))$, where $b(x, y)$ is the point where the ray from $x$ through $y$ hits the boundary of $\Omega$.  This defines a 
metric on $\Omega$, which is Finslerian and projective.  We recall that a Finslerian metric on $\Omega$ (or, more generally, on a differentiable manifold) is determined by a norm on each tangent space in such a way that  the distance between two point in $\Omega$ is equal to the infimum of lengths of paths joining them,  where the length of a path is computed by integrating  along it the norms of vectors tangent to this path. This 
norm function is usually called the \emph{Minkowski functional} of the metric. We also recall that a metric on a subset of Euclidean space is said to be 
projective if Euclidean straight lines are geodesics for that metric.

  The open unit 
disc in $\mathbb{R}^n$ equipped with its Hilbert metric $H(x, y)$ is a prominent example of a projective metric, since it is
Klein's model of hyperbolic $n$-space, and it was the motivating example for Hilbert when he defined his metric for more general convex sets.

The value $H(x, y)$ of the Hilbert metric  on $\Omega$ can be written, for $x\not= y$,  as 
\[
\log \frac{|x-b(x,y)||y-b(y,x)|}{|y-b(x,y)||x-b(y,x)|} =  \log \frac{|x-b(x,y)|}{|y-b(x,y)|}
+ \log \frac{|y-b(y,x)|}{|x-b(y,x)|}.
\]

P. Funk~\cite{F} looked at the first term of the right hand side of the above equation as a kind of metric, even though it is not symmetric in $x$ and $y$. This is now called the \emph{Funk metric}.  The reader is referred to Funk's paper \cite{F} and to the papers \cite{PT1, PT2} for some historical and technical background on the Funk metric.

More generally, given a set $X$, we can consider
 functions $\delta: X \times X \rightarrow \mathbb{R}_+ \cup \{ \infty \} $ satisfying the following two properties:
\begin{enumerate}
\item $\delta(x,x) = 0$ for all $x$ in $X$ ;
\item $\delta(x, z) \leq \delta(x, y) + \delta(y, z)$ for all $x,y$ and $z$ in $X$.
\end{enumerate}

In \cite{PT1}, such a function $\delta$ is named \emph{weak metric}. Note that in this definition  neither the symmetry $\delta(x, y) = \delta(y,x)$ nor the nondegeneracy $\big(\delta(x, y) = 0 \Rightarrow x=y\big)$ are assumed. A weak metric may be Finslerian in the sense that it is induced by a \emph{weak norm} on each tangent space in the same way as a genuine metric can be Finslerian. Here, a \emph{weak norm} on a vector space is a function $\xi\mapsto \Vert \xi\Vert$ satisfying
 \begin{enumerate}
\item $\Vert 0\Vert =0$ ;
\item  $\xi\mapsto \Vert \xi\Vert$ is convex.
\end{enumerate}
 
 For brevity, we shall call a weak metric a \emph{metric}, and when such a function does or does not satisfy the other axioms satisfied by a metric, we shall mention it explicitly whenever this is needed.

We recall that Hilbert's Fourth Problem from the collection of mathematical problems he presented in 1900 at the Second International Congress of Mathematicians in Paris, is entitled: ``Problem of the straight line as the shortest distance between two points."  Hilbert elaborates on this statement in \cite{Hilbert-Problems}, and he also places the problem in a historical perspective. Mathematicians however agree on the fact that Hilbert's formulation of the problem is rather vague, which makes this problem (like several others in Hilbert's collection), admit several different precise formulations and therefore several possible solutions. A possible (and probably the most common) formulation of that problem is in two parts, as follows: (1) to characterize the metrics on subsets of Euclidean space for which the Euclidean straight lines are geodesics, and (2) to study such metrics individually. (We note by the way that the axioms of a metric space as we intend it today had not been formulated at the time Hilbert proposed his problems; they were given later on by Maurice Fr\'echet in his thesis (1907).)

Herbert Busemann, who spent much of his career  engaged with that problem, gave it the following formulation (see \cite{Busemann1974}), which is close to the one we gave above: ``The fourth problem concerns the geometries in which the ordinary lines, i.e. lines of an n-dimensional (real) projective space $\mathbb{P}^n$ or pieces of them are the shortest curves or geodesics. Specifically, Hilbert asks for the construction of these metrics and the study of the individual geometries."  

For the Funk and the Hilbert metrics in hyperbolic and in spherical geometry, we shall see the non-Euclidean geodesics are geodesics for these metrics. It is then natural to address the question of considering Hilbert Problem IV in these non-Euclidean settings, that is, the problem of characterizing and of studying individually the metrics on subsets of the sphere and on subsets of hyperbolic space for which the spherical and the hyperbolic straight lines respectively are geodesics. We shall see that these problems amount to the Euclidean problem. A related fact is that despite of the lack of linear structure in the hyperbolic and spherical geometries, when it comes to the geometry of Hilbert metrics, one can still capture the projective geometry and the so-called incidence geometry in a manner almost identical to the one in the Euclidean situation, from the viewpoint of convex geometry.
We further remark that the projective structure only appears in the geometry of Hilbert metrics, and not in that of Funk metrics.  

The authors would like to thank Norbert A'Campo for sharing his enthusiasm and ideas.

\section{The Funk metric in Euclidean space} \label{s:Funk}
First we collect some known facts about the Funk and Hilbert metrics defined on convex subsets of Euclidean spaces.   We set the presentation in \cite{PT1} as our reference for the Funk 
and Hilbert metrics, and we also refer to the first part of the paper \cite{Y3}.

Let  $\Omega$ be an open convex subset of Euclidean space $\mathbb{R}^n$. Note that we allow $\Omega$ to be unbounded.

There are three different descriptions of the Funk metric.  
The first one is the original definition which we already referred to:
\[
F_1(x, y) = \log \frac{d(x, b(x,y))}{d(y,  b(x,y))},
\]
 for $x\not= y$ in $\Omega$, the point $b(x,y)$ being the intersection point of the Euclidean ray $\{x+t \xi_{xy}: t>0\}$ 
from $x$ though $y$ with the boundary $\del \Omega$ when such an intersection point exists, $\xi_{xy}$ being the unit tangent vector in $\mathbb{R}^n$ pointing from $x$ to $y$. In the case where the Euclidean ray $\{x+t \xi_{xy}: t>0\}$ is contained in $\Omega$ (and therefore does not intersect the boundary), we set the distance $F_1(x, y)$ to be 0. This makes the function $(x,y)\mapsto F_1(x,y)$ defined on the whole set $\Omega\times \Omega$ and continuous with respect to the topology of $\mathbb{R}^n$.

The second description is the variational interpretation of the above value  using the geometry of supporting hyperplanes; we set:
\[
F_2 (x,y) =  \sup_{\pi \in {\mathcal{P}}} \log \frac{d(x, \pi)}{d(y,  \pi)},
\]
where $\mathcal{P}$ is the set of all supporting hyperplanes of $\Omega$. This is given in \cite{Y3}. In  the literature, a variational characterization of Hilbert metrics appears in the work of Nussbaum (\cite{Nu}, 1988), where the supporting hyperplanes are treated as the elements of the dual space; note that the Funk metric was not mentioned in that work.

Finally,  the Finsler structure $p_{\Omega, x}(\xi)$ of the Funk metric is given by the following function (the Minkowski functional) on vectors $\xi$ at each tangent space to $\Omega$ at $x$:
\[
p_{\Omega, x}(\xi) = \sup_{\pi \in {\mathcal{P}}} \frac{ \langle \nu_{\pi}(x), \xi \rangle}{d(x, \pi)}.
\]   
This is a weak norm on each tangent space which is defined so that the Funk distance is described as the infimum of lengths of 
curves:
\[
F_3 (x,y) = \inf_{\sigma} \int_a^b p_{\Omega, \sigma (t)} (\dot{\sigma}(t)) dt,
\]
 the infimum being taken over the piecewise $C^1$-curves $\sigma$ in $\Omega$ with $\sigma(a) = x$ and $\sigma(b) = y$. 

We emphasize that 
for a convex domain $\Omega \subset \mathbb{R}^n$, the three quantities $F_1(x,y)$, $F_2(x,y)$, $F_3(x,y)$ are  all equal to each other, and we set
\[
F(x, y) := F_1(x,y) = F_2(x,y)= F_3(x,y)
\]
for every $x$ and $y$ in $\Omega$. 

The aim of this paper is to consider Funk-like metrics in non-Euclidean geometries: hyperbolic and spherical. Formally, the exposition is very similar in the two cases. We shall first give an exposition of the theory in  the case of the $n$-dimensional hyperbolic space and we shall then mention the changes needed for the case of spherical geometry.
 
\section{The hyperbolic and spherical Funk metrics}

\subsection{The hyperbolic Funk metric}\label{s:Funk-H}

Given an open convex set $\Omega$ in ${\mathbb H}^n$, we shall define a 
Funk-type metric (which we shall call the \emph{Funk metric} of $\Omega$), and provide three descriptions of it 
corresponding to $F_1$, $F_2$ and $F_3$ of the Euclidean Funk 
metric case. 

We shall also study the geodesics of this metric. We recall that a path $s:[a,b] \rightarrow (X, d)$ in a metric space $(X, d)$ is said to be 
geodesic when for any $a<t<b$, the equality $d(s(a),s(t)) + d(s(t), s(b)) = d(s(a), s(b))$
is satisfied.

We start by representing the convex set $\Omega$ as 
$
\cap_{\pi(b) \in {\mathcal{P}}} H_{\pi (b)}
$
where $H_{\pi(b)}$ is the (open) half space bounded by a 
totally geodesic hyperplane touching $
\Omega$ at the boundary point $b$ and containing the convex set $
\Omega$. In analogy with 
the Euclidean situation, we call these submanifolds  {\it supporting hyperplanes} of $\Omega$.   
The index set $\mathcal{P}$ is the set of all supporting hyperplanes
of $\Omega$. That for every boundary point $p$ there exists a 
supporting hyperplane $\pi(b)$ follows from the 
definition of convexity of $\Omega$.  In general, there can be more than one supporting hyperplane
of $\Omega$ at $p \in \del \Omega$. 
For later use, we denote by ${\mathcal{P}}(b)$ the set of supporting hyperplanes at $b \in \del \Omega$.
We denote by $d$ the hyperbolic metric in $\mathbb{H}^n$.  \\

Given two distinct points $x$ and $y$ in a convex set $\Omega$, we denote by $R(x,y)=\{ \exp_x (t \xi_{xy}) \,\, | \,\, 
t>0\}$ the geodesic ray starting at $x$ and passing through $y$ and where, as in the Euclidean case, $\xi_{xy}$ is the unit tangent 
vector at $x$ of the arc-length parameterized geodesic in $\mathbb{H}^n$ connecting $x$ and $y$. 

\begin{definition} \label{d:Funk-h}
For a pair of points $x$ and $y$ in $\Omega \subset {\mathbb H}^n$, 
 the Funk (asymmetric) distance from $x$ to $y$ is defined by 
\begin{equation*}
F(x, y)= 
\begin{cases}  \displaystyle \log \frac{\sinh d(x, b(x,y))}{\sinh d(y,  b(x,y))} & \text{if $x\not=y$ and $R(x,y)\cap \partial \Omega \not=\emptyset$,}
\\
0 &  \text{otherwise}
\end{cases}
\end{equation*}
where the point $b(x,y)$ is the intersection with the boundary $\del 
\Omega$ of the hyperbolic geodesic ray $R(x,y)$ from $x$ though $y$. 
\end{definition}
  
We shall see below that the function $F$ satisfies indeed the triangle inequality.

We will consider only the case where the ray $R(x,y)$ is not contained in $\Omega$. The other case can be dealt with easily.

We first recall a classical trigonometric identity on the hyperbolic plane ${\mathbb H}^2$. For a given hyperbolic right triangle $\triangle(A,B,C)$
with angles $\alpha, \beta,\gamma$ with $\gamma = \pi / 2$ and with side lengths $a,b$ and $c$  
opposite to the vertices $A, B$ and $C$ respectively, we have
\[
\sinh b = \sinh c \sin \beta.  
\]
The formula is a special case of the Sine Rule which is recalled in the Appendix.

Note that a Euclidean right triangle with corresponding labelling would satisfy $b = c \sin \beta$  and we have here an instance of a correspondence which often occurs between the Euclidean and the hyperbolic trigonometric formulae where the hyperbolic formulae are obtained by replacing the side lengths by the hyperbolic sines of these lengths, and similar transformations. (See as an example the Sine Rule in the Appendix.) Choosing a point $A'$ on the side  $c$ and letting  $C'$ be its nearest point projection on the side $a$, we have another right triangle $\triangle(A',B,C')$ with angles
$\alpha', \beta,\gamma' $, with $\gamma'= \pi/2$, and with side lengths $a',b,c'$  
opposite to the vertices $A', B$ and $C'$ respectively, satisfying 
\[
\sinh b' = \sinh c' \sin \beta.  
\]
As the ratios $\sinh b / \sinh c$ and $ \sinh b' / \sinh c' $ are equal to $\sin \beta$, we shall say that the two triangles $\triangle(A,B,C)$ and $\triangle(A',B,C')$ are {\it similar} with the side lengths being weighted by the $\sinh$ function.

\begin{figure}[!hbp]
 \centering
  \psfrag{B}{\small $B$}
   \psfrag{A}{\small $A$}
    \psfrag{C}{\small $C$} \psfrag{a}{\small $a$} \psfrag{b}{\small $b$} \psfrag{c}{\small $c$} \psfrag{D}{\small $A'$} \psfrag{E}{\small $E$} \psfrag{F}{\small $C'$} \psfrag{e}{\small $a'$} \psfrag{f}{\small $c'$} \psfrag{v}{\small $\beta$}
     \includegraphics[width=0.35\linewidth]{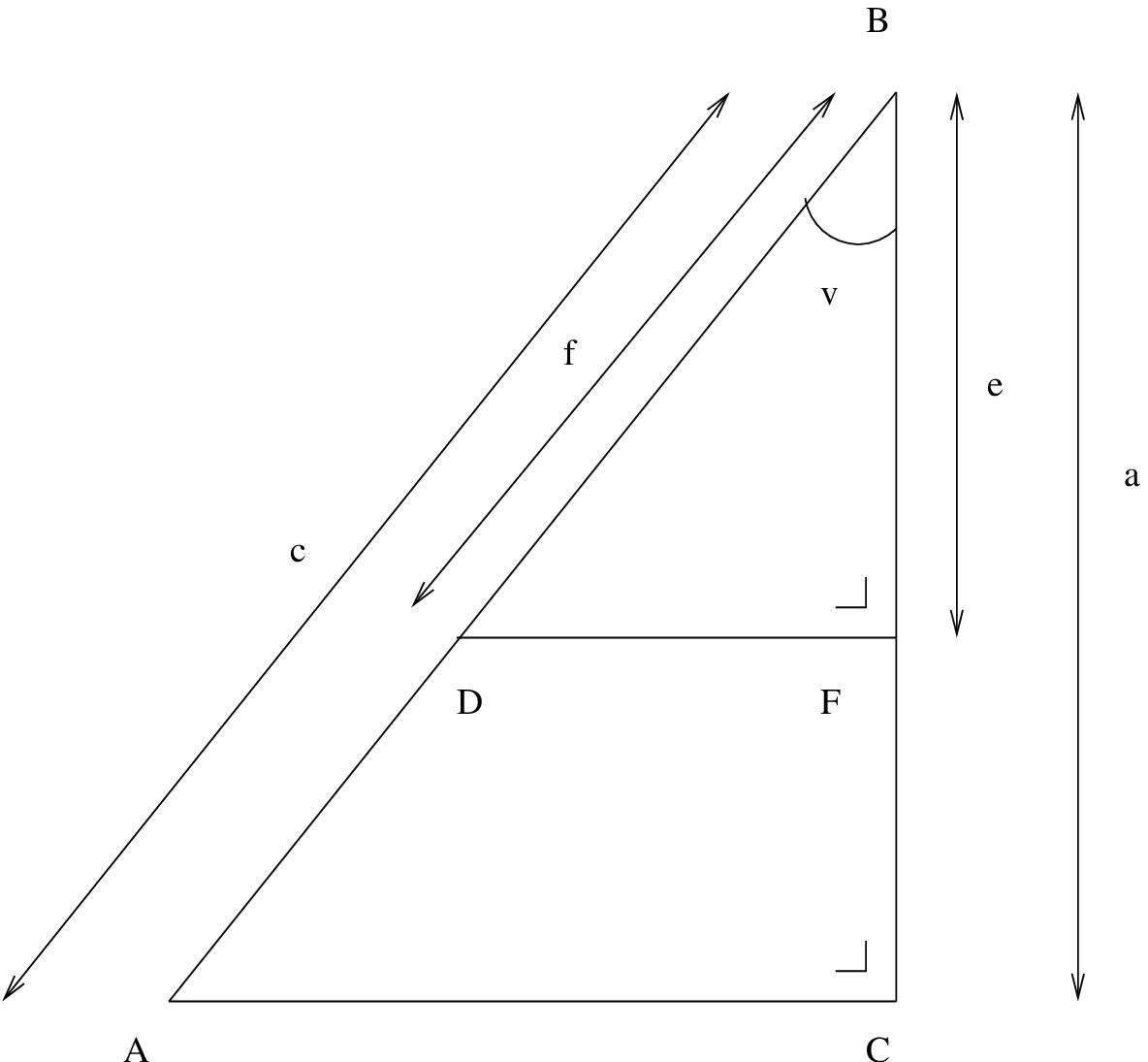} \caption{\small {}} \label{Funk1} \end{figure}

We now let $\pi_0$ be a supporting hyperplane for $\Omega$ at $b(x,y)$, namely, $\pi_0 \in {\mathcal{P}}(b(x,y))$.
We note the similarity (in the above sense) between the right triangles $\triangle (x, \Pi_{\pi_0}(x), b(x,y))$
and $\triangle (y, \Pi_{\pi_0}(y), b(x,y))$, where $\Pi_{\pi_0}(p)$ is the foot of the point $p$ on the hyperplane $\pi_0$, or, putting it differently, where $\Pi_{\pi_0} : \mathbb{H}^n \rightarrow \pi_0$ is the nearest point projection map. The two triangles $\triangle (x, \Pi_{\pi_0}(x), b(x,y))$ and $\triangle (y, \Pi_{\pi_0}(y), b(x,y))$ lie in a hyperbolic plane ${\mathbb H}^2$ isometrically embedded in ${\mathbb H}^n$ which is uniquely determined by the three points $x, \Pi_{\pi_0}(x), b(x,y)$ it contains.  

Thus, we have
\[
\log \frac{\sinh d(x, b(x,y))}{\sinh d(y,  b(x,y))} = \log \frac{\sinh d(x, \pi_0)}{\sinh d(y,  \pi_0)}.
\]
By the above similarity 
property of triangles, the right hand side of the equality is independent of the 
choice of $\pi_0$ in ${\mathcal{P}}(b(x, y))$.

Now by the convexity of $\Omega$, the quantity $F(x, y)$ in Definition \ref{d:Funk-h} can be characterized variationally as follows. 
Define $T(x,\xi, \pi)$ by $\pi \cap R(x,y)$ with $\pi \in {\mathcal{P}}$. 
Consider the case $\xi = \xi_{xy}$.  When the hyperplane $\pi$ supports $\Omega$ at $b(x, y)$, we have $T(x,\xi_{xy}, \pi) = b(x, y)$
and otherwise  (when $\pi \notin {\mathcal{P}}(b(x, y))$)  the point $T(x, \xi_{xy}, \pi)$ lies outside $\Omega$.  
When $\pi \notin {\mathcal{P}}(b(x,y))$,  by the similarity property between the triangles
$\triangle (x, F_{\pi}(x), T(x, \xi_{xy}, \pi))$ and $\triangle (y, F_{\pi}(y), T(\xi_{xy}, \pi))$ again, we have 
\[
\frac{\sinh d(x, \pi)}{\sinh d(y, \pi)} =\frac{\sinh d(x, T(x,\xi_{xy},\pi))}{\sinh d(y, T(x,\xi_{xy}, \pi))}.
\]
Note that the closest point to $x$ along the ray $\{ \exp_x t \xi_{xy} \,\, | \,\,  t\geq 0 \}$ of the form $T(x,\xi_{xy}, \pi)$ is $b(x, y)$.    
This in turn says that  a hyperplane $\pi$ which supports $\Omega$ at $b(x,y)$ maximizes the ratio $d(x, T(x,\xi_{xy}, 
\pi))/d(y, T(x,\xi_{xy}, \pi))$ among all the elements of ${\mathcal{P}}$; that is,

\[
 \log \frac{\sinh d(x, b(x,y))}{\sinh d(y,  b(x, y))} = \sup_{\pi \in {\mathcal{P}}} \log \frac{\sinh d(x, \pi)}{\sinh d(y,  \pi)}. 
\]

Hence we have the following characterization of the Funk metric $F$ for $\Omega \subset {\mathbb H}^n$:
\begin{theorem} The Funk metric on a convex subset $\Omega \subset {\mathbb H}^n$ has the following variational formulation:
\[
F(x, y) = \sup_{\pi \in {\mathcal{P}}} \log \frac{\sinh d(x, \pi)}{\sinh d(y,  \pi)}.
\]    
\end{theorem}

With this formulation, one can readily see that $F(x, y)$ satisfies the triangle inequality, for
 \begin{eqnarray*}
F(x, y) + F(y, z)  &=&  \sup_{\pi \in {\mathcal{P}}} \log \frac{\sinh d(x, \pi)}{\sinh d(y, \pi)} +  \sup_{\pi \in {\mathcal{P}}} \log \frac{\sinh d(y, \pi)}{\sinh d(z, \pi)} \\
 & \geq &  \sup_{\pi \in {\mathcal{P}}} \Big( \log \frac{\sinh d(x, \pi)}{\sinh d(y, \pi)} +   \log \frac{\sinh d(y, \pi)}{\sinh d(z, \pi)} \Big) \\& =  & \sup_{\pi \in {\mathcal{P}}} \log \frac{\sinh d(x, \pi)}{\sinh d(z, \pi)} \\ & = & F(x, z)
\end{eqnarray*}

Note that the triangle inequality becomes an equality when 
\[
\sup_{\pi \in {\mathcal{P}}} \log \frac{\sinh d(x, \pi)}{\sinh d(y, \pi)} +  \sup_{\pi \in {\mathcal{P}}} \log \frac{\sinh d(y, \pi)}{\sinh d(z, \pi)} 
  =   \sup_{\pi \in {\mathcal{P}}} \Big( \log \frac{\sinh d(x, \pi)}{\sinh d(y, \pi)} +   \log \frac{\sinh d(y, \pi)}{\sinh d(z, \pi)} \Big)
\] 
is satisfied.  For this to occur, we only need to have ${\mathcal{P}}(b(x, y)) \cap {\mathcal{P}}(b(y,z)) \neq \emptyset$.  Indeed, let $\pi_0$ be an element of the set  ${\mathcal{P}}(b(x, y)) \cap {\mathcal{P}}(b(y,z)) \neq \emptyset$. Then 
the boundary points $b(x,y)$ and $b(y,z)$ share the same supporting hyperplane $\pi_0$, and therefore  
\[
\sup_{\pi \in {\mathcal{P}}} \log \frac{\sinh d(x, \pi)}{\sinh d(y,  \pi)}
= \log \frac{\sinh d(x, \pi_0)}{\sinh d(y, \pi_0)},\]
\[
\sup_{\pi \in {\mathcal{P}}} \log \frac{\sinh d(y, \pi)}{\sinh d(z,  \pi)}
= \log \frac{\sinh d(y, \pi_0)}{\sinh d(z, \pi_0)}
\]
and 
\[
\sup_{\pi \in {\mathcal{P}}} \log \frac{\sinh d(x, \pi)}{\sinh d(z,  \pi)}
= \log \frac{\sinh d(x, \pi_0)}{\sinh d(z, \pi_0)}
\]
inducing the equality. The observation is summarized as in the following proposition. For $x$ and $y$ in $\Omega$, we denote, as before, by $R(x)$ the geodesic ray starting at $x$ and passing through $y$.

\begin{proposition} Let $\Omega$ be an open convex subset of ${\mathbb  H}^n$ such that $\del \Omega$ contains some hyperbolic geodesic segment $\sigma:[p, q] \rightarrow {\mathbb H}^2$ which we denote by $\overline{pq}$ and let $x$ and $z$ be two distinct points in $\Omega$ such that $R(x,y) \cap \overline{pq} \neq \emptyset$.  Let $\Omega'$ be the intersection of $\Omega$ with the hyperbolic plane ${\mathbb H}^2$ in ${\mathbb H}^n$ containing $\{x\} \cup \overline{pq}$.  Then, for any point $y$ in $\Omega'$ satisfying $R(x,y)  \cap \overline{pq} \neq \emptyset$ and $R(y,z)  \cap \overline{pq} \neq \emptyset$, we have $F(x, y) + F(y,z) = F(x,z)$. 
\end{proposition}

A notable situation when one has ${\mathcal{P}}(b(x, y)) \cap {\mathcal{P}}(b(y,z)) \neq \emptyset$ is when
$x,y$ and $z$ are collinear, meaning that they lie on a common geodesic, with $y$ lying between $x$ and $z$.  This in turn says that
the hyperbolic geodesics are Funk geodesics, or -- as Hilbert would say --  that the Funk metric is projective. This result is a hyperbolic analogue of Corollary 8.2 of \cite{PT1}.

On the other hand, when $\pi_0$ is in the set  ${\mathcal{P}}(b(x, y)) \cap {\mathcal{P}}(b(y,z))$ and the three points $x,y,z$ do not lie on a geodesic,  the concatenation of the 
geodesic segments $\overline{xy}$ and $\overline{yz}$ is also a Funk geodesic, a situation occurring when 
the boundary set $\del \Omega$ contains a hyperbolic geodesic segment; a statement which is a hyperbolic analogue of Corollary 8.4 of \cite{PT1}.

We next consider the complementary situation where  ${\mathcal{P}}(b_1) \cap {\mathcal{P}}(b_2) = \emptyset$ for any 
pair of distinct points $b_i \in \del \Omega$. Geometrically this characterizes strict convexity of the domain $\Omega$, namely
the boundary $\del \Omega$ contains no closed line segment.  From the preceding argument, it follows that the only  
way the equality for the triangle inequality occurs is when the three points $x,y$ and $z$ are collinear. Hence,
for strictly convex domains, the Funk geodesics consist of line segments only. Equivalently, given a pair of distinct points in $\Omega$, there is a unique Funk geodesic connecting them. This corresponds to Corollary 8.8 of \cite{PT1}.
    
We summarize this in the following

\begin{proposition}\label{prop:geod-F}
Let $\Omega$ be an open convex  subset of $\mathbb{H}^n$ and let $F(x,y)$ be its Funk metric. Then,
\begin{enumerate}
\item the hyperbolic geodesics of $\Omega$ are also Funk geodesics;
\item the Funk geodesics of $\Omega$ are hyperbolic geodesics if and only if  $\Omega$ is strictly convex, that is, if its boundary does not contain any nonempty open hyperbolic segment.
\end{enumerate}
\end{proposition}

There is another proof of Proposition \ref{prop:geod-F}  that gives at the same time the 
triangle inequality for the Funk metric. It imitates the classical proofs of the triangle inequality given in the case of the Euclidean Funk geometry that are given in \cite{Busemann1953} p. 158 and \cite{Zaustinsky} p. 85. We present it in the appendix of this article.

Let us also note the following
\begin{proposition}
Any hyperbolic geodesic segment starting at a point $x$ in $\Omega$ and ending at a point on $\partial\Omega$ is (the image of) a geodesic ray for the Funk metric on $\Omega$.
\end{proposition}

\begin{proof}
This follows from the formula defining the Funk metric (Definition \ref{d:Funk-h}) and the fact that the hyperbolic geodesics in $\Omega$  are Funk geodesics.
\end{proof}

\begin{example}
Let us consider the particular example of a Funk metric, where the open convex set $\Omega$ is an ideal triangle of ${\mathbb H}^2$.  
We choose the case of the ideal triangle because such a triangle exists only in hyperbolic geometry. We can model the triangle to be the region in the upper half plane bounded by the $y$-axis, which we call $\pi_1$, the line $\{x=1\}$, which we call $\pi_2$ and the semi-circle $\pi_3$ connecting the origin and $(1, 0)$ and which is perpendicular to the real line. The lines $\pi_1,\pi_2,\pi_3$ are also the supporting hyperplanes of $\Omega$. Without loss of generality, we suppose that two distinct points $x_1$ and $x_2$ in $\Omega$ are located so that the hyperbolic geodesic ray from $x_1$ through $x_2$ hits the $y$-axis $\pi_1$ at $b(x_1, x_2)$, and hence the hyperbolic Funk distance
\[
F(x_1, x_2) = \max_{\pi_1, \pi_2, \pi_3} \log  \frac{\sinh d(x_1, \pi_i)}{\sinh d(x_2, \pi_i)} 
\]     
is realized by $\log \displaystyle \frac{\sinh d(x_1, \pi_1)}{\sinh d(x_2, \pi_1)}$.
Then an elementary calculation gives an explicit value of the Funk distance as 
\[
F(x_1, x_2) = \log \Big( \frac{1- m_1^2}{1- m_2^2 } \cdot \frac{m_2}{m_1} \Big)
\]
where $m_i$ is the absolute value of the slope of the (Euclidean) line segment connecting $x_i$ and the (hyperbolic) foot of $x_i$ on the $y$-axis $\pi_1$. 
  
\end{example}

We next consider the infinitesimal linear structure of the Funk metric $F$, by identifying it with a Finsler norm on the tangent spaces. We first recall the Euclidean setting. In this setting, the Funk metric is induced by a Finsler structure, the \emph{tautological weak Finsler structure} in the sense of~\cite{PT1}, given by the following Minkowski functional:
\[
p_{\Omega, x}(\xi) = \sup_{\pi \in {\mathcal{P}}}  \frac{\| \xi\|}{d(x, T(x, \xi, \pi))}.
\]
In this formula, the supremum is achieved when the supporting hyperplane $\pi$ supports $\Omega$ at
the point where the ray $\{x+ t \xi \}$ meets the boundary set $\partial \Omega$. 
Using similarity of Euclidean triangles, this can be written as
\[
p_{\Omega, x}(\xi) = \sup_{\pi \in {\mathcal{P}}}  \frac{\langle \xi, \eta_\pi \rangle}{d(x, \pi)} 
\] 
where $\eta_\pi$ is the unit tangent vector at $x$ with direction opposite to the gradient vector of the functional $d(., 
\pi)$. We note that the gradient vector field's integral curves are the geodesics which meet the supporting hyperplane perpendicularly.

The infinitesimal linear structure of the Funk metric is obtained by linearizing the following expressions for a fixed $\pi \in {\mathcal P}$,
\[
 \log \frac{d(x, T(x, \xi, \pi))}{d(\alpha(t), T(x, \xi, \pi))} \mbox{ or equivalently } 
 \log \frac{d(x, \pi)}{d(\alpha(t), \pi)},
\]
 where $\alpha:[0,\infty) \rightarrow {\mathbb R}^n$ is a geodesic ray  with $\alpha(0)=x, \alpha(1) = y$ and $\|\alpha'(t)\|= \mbox{const.}$ and where
$T(x, \xi, \pi)$ is the point where the geodesic ray and the hyperplane $\pi$ intersect.  We then have 
\[
\frac{\| \alpha'(0) \|}{d(x, T(x, \alpha'(0), \pi))}
\mbox{ or equivalently } 
\frac{\langle \alpha'(0), \eta_\pi \rangle}{d(x, \pi)}. 
\]

In turn, In the hyperbolic space we identify the value of the Minkowski functional for the Funk metric $F$ as
\begin{equation}\label{fo:1}
p_{\Omega, x}(\xi) = \sup_{\pi \in {\mathcal{P}}} \frac{\cosh d(x,  T(x, \xi, \pi) ) }{\sinh d(x,  T(x, \xi, \pi))} 
\| \xi \|,
\end{equation}
or equivalently as  
\begin{equation}\label{fo:2}
p_{\Omega, x}(\xi) = \sup_{\pi \in {\mathcal{P}}} \frac{\cosh d(x, \pi) }{\sinh d(x, \pi)} \langle \eta_{\pi}(x), \xi \rangle
\end{equation}
where $\eta_\pi$ is the unit tangent vector at $x$ whose 
direction is opposite to the one of the gradient vector of the functional $d(*, 
\pi)$ at $x$.  In this formula, it is easy to see that the supremum is achieved when the point $T(x, \xi, \pi)$ coincides with a boundary point $b \in \partial \Omega$, namely $\pi \in \mathcal{P}(T(x, \xi, \pi))$.

The two representations come from the linearizations of 
\[
 \log \frac{\sinh d(x, T(x, \xi, \pi))}{\sinh d(\alpha(t), T(x, \xi, \pi))} \mbox{ and } 
 \log \frac{\sinh d(x, \pi)}{\sinh d(\alpha(t), \pi)} 
\]
respectively, where $\alpha:[0,\infty) \rightarrow {\mathbb H}^n$ is a geodesic ray  with $\alpha(0)=x, \alpha(1) = y$ and $\|\alpha'(t)\|= \mbox{const.}$ and where$T(x, \xi, \pi)$ is the point where the geodesic ray and the hyperplane $\pi$ intersect.

In either representation, it is easy to see that the functional 
is convex in $\xi \in T_x \mathbb{R}^n$, since the functional $ p_{\Omega, x}$ is
convex  (linear in particular)  in $\xi$ for each fixed $\pi \in {\mathcal{P}}$, and since by taking the $\sup$ over $\pi $,  the 
convexity is preserved.  

Alternatively one can see the convexity of the \emph{indicatrix} $C(x, 
\Omega)$, that is, the set of vectors with norm equal to 
one:
\[
C(x, \Omega) := \{ \xi \in T_x {\mathbb H}^n \,\, | \,\, p_{\Omega, x}(\xi) = 1 \}, 
\]
by noting that for $\xi_1$ and $\xi_2$ in $C(x, \Omega)$, we have
\begin{eqnarray*}
p_{\Omega, x}(\frac{\xi_1 + \xi_2}{2}) &=& \sup_{\pi \in {\mathcal{P}}} \frac{\cosh d(x, \pi) }{\sinh d(x, \pi)} \langle \nu_{\pi}(x), 
\frac{\xi_1 + \xi_2}{2} \rangle \\ &\leq &
\frac12 \sup_{\pi \in {\mathcal{P}}} \frac{\cosh d(x, \pi) }{\sinh d(x, \pi)} \langle \nu_{\pi}(x), \xi_1 \rangle + \frac12 \sup_{\pi \in {\mathcal{P}}} \frac{\cosh d(x, \pi) }{\sinh d(x, \pi)} \langle \nu_{\pi}(x), \xi_2 \rangle \\
& = & \frac12 + \frac12 = 1.
\end{eqnarray*}
This says $(\xi_1 + \xi_2)/2$ lies inside the indicatrix $C(x, \Omega)$, hence the unit ball of the norm $p_{\Omega, x}$ is a convex set in $T_x \Omega$.

We now show that the infimum, among all piecewise $C^1$-paths $\gamma$ with given endpoints, of the length computed with the Finsler norm
$p_{\Omega, x}$
coincides with the Funk metric. Namely:

\begin{theorem}
The Finsler distance $d(x, y)$ induced by the Minkowski functional $p_{\Omega, x}$ (\ref{fo:1}) or (\ref{fo:2}) coincides with the hyperbolic Funk metric $F$ on $\Omega
\subset {\mathbb H}^n$.
\end{theorem}
 
\begin{proof}
For a given pair of points $x$ and $y$ in $\Omega$, let $\alpha(t) 
$ be the hyperbolic geodesic ray $\exp_x t \xi_{xy}$ from $x$ through $y$, and let 
$b(x, y) = T(x, \xi_{xy}, \pi_{b(x,y)})$ be the point where this geodesic 
ray hits the boundary $\partial \Omega$ (with the 
notation we have been using in the Euclidean case). Then the $F$-length of the curve $
\alpha$ from  $x$ to $y$ is 
\[
L(\alpha) = \int_0^1 p_{\Omega, \alpha(t)} (\alpha'(t)) dt = \log 
\frac{\sinh d(x, b(x, y))}{\sinh d (y, b(x, y))} = F(x, y).
\]  
Thus, the Finsler distance $d(x, y) = \inf_\gamma L(\gamma)$ is bounded above by the Funk-
type distance $F(x, y)$. 

On the other hand, consider the convex set in ${\mathbb H}^n$ 
bounded by the supporting hypersurface $\pi_{b(x, y)}$ alone.  This 
set has its own Funk metric,
\[
F_{\pi_{b(x, y)}}(s, t) := \log \frac{\sinh d(s, \pi_{b(x, y)})}{\sinh d (t, \pi_{b(x, y)})}
\]
There is a simple comparison $F(s, t) \geq F_{\pi_{b(x, y)}}(s, 
t)$, as $F(s, t)$ is the supremum of $\displaystyle \log \frac{\sinh d(s, \pi)}{\sinh d (t, \pi)}$ over all the supporting hypersurfaces 
of which $\pi_{b(x, y)}$ is one.  Hence for the pair $(x, y)$, 
the value  
\[
F_{\pi_{b(x, y)}}(x, y) = \log \frac{\sinh d(x, \pi_{b(x, y)})}{\sinh d (y, \pi_{b(x, y)})}
\]
provides a lower bound for $F(x, y)$.

As we have the equality 
\[
\log \frac{\sinh d(x, b(x, y))}{\sinh d (y, b(x, y))}=
\log \frac{\sinh d(x, \pi_{b(x, y)})}{\sinh d (y, \pi_{b(x, y)})},
\]
we conclude that $d(x, y) = F(x, y)$.
\end{proof}
 
We end this section by the following result, which contrasts with the case of Euclidean geometry:
 \begin{proposition}\label{prop:nonconvex}
 Given a point $x$ in $\Omega$, the Funk distance function $F(x,y)$ is not convex in $y$.
 \end{proposition}
 \begin{proof}

Consider a point $x$ in $\Omega$ and a constant speed geodesic segment $\alpha: [0,1] \rightarrow 
{\mathbb H}^n$ with $\alpha(0)=y$ where $y$ is any point in $\Omega$.  For a fixed supporting hyperplane $\pi$ of $\Omega$, we denote by $\eta_\pi$ the vector field that is the vector opposite to the gradient vector field of the functional $d(., \pi)$. Then the first and second derivative of the Funk distance function are given as:
\[
\frac{d}{dt} \log \frac{\sinh d(x, \pi)}{\sinh d(\alpha(t), \pi)} \ = 
\frac{\cosh d(\alpha(t), \pi)}{\sinh d(\alpha(t), \pi)}
 \langle \dot{\alpha}(t), \eta_\pi (\alpha(t)) \rangle
\]
 and 
\begin{eqnarray*}
\frac{d^2}{dt^2} \log \frac{\sinh d(x, \pi)}{\sinh d(\alpha(t), \pi)} \Big|_{t=0} 
 & = &\frac{1}{\sinh^2 d(y, \pi)} \langle \dot{\alpha}(0), \eta_\pi(x) \rangle^2 \\
 & &
+ \frac{\cosh d(y, \pi)}{\sinh d(y, \pi)} \Big( \frac{d}{dt} \langle \dot{\alpha}(t), \eta_\pi (\alpha(t)) \rangle \Big|_{t=0} \Big). 
\end{eqnarray*}

Note that the quantity $\langle \dot{\alpha}(t), \eta_\pi (\alpha(t)) \rangle $ is equal to 
$\| \dot{\alpha}(t) \| \cos \theta(t)$ where $\theta(t)$ is the angle between the velocity vector 
$\dot{\alpha}(t)$ and the unit vector $\eta_\pi(\alpha(t))$, and that the speed $\| \dot{\alpha}(t) \| $ is constant over time.  Now the negative sectional  curvature of ${\mathbb H}^n$ implies that the angle $\theta(t)$ is increasing in $t$, which in turns implies  
\[
\frac{d}{dt} \langle \dot{\alpha}(t), \eta_\pi (\alpha(t)) \rangle  < 0.
\]   
Therefore the value of the second derivative cannot be of a definite sign, proving the claim. 
\end{proof}
 
It is known (\cite{Y3}) that in contrast, the Euclidean Funk metric $F(x, y)$ {\it is} convex in $y$.  
 
\subsection{The Funk metric in spherical geometry}

We shall use the language we are used to in Euclidean and hyperbolic geometry, namely, we shall call a \emph{hyperplane} a complete totally geodesic codimension-one subspace. Likewise, a (closed or open) connected component in $S^n$  bounded by a hyperplane is  called a \emph{half-space}.

The definition of a convex set on the sphere is more delicate than the one in Euclidean or in hyperbolic space, because given any distinct two points on the sphere, there are two distinct geodesics (arcs of great circle) joining them. Even if we insist on geodesics of shortest length, for some pairs of points of the sphere (namely, for points which are diametrically opposite), there are two distinct geodesics of shortest length. For this reason, we give the following

\begin{definition}
A subset $\Omega$ of the sphere is said to be \emph{convex} if $\Omega$  is contained in an open half-space and if every pair of points in $\Omega$ can be connecetd by a geodesic which is contained in $\Omega$.
\end{definition}

\subsection{Definition}
Given a convex set $\Omega$ in $S^n$, represent it as $\cap_{\pi(b) \in {\mathcal P}} H_{\pi(b)}$ where $H_{\pi(b)}$ is a half-space bounded by a hyperplane $\pi(b)$ tangent to $\partial \Omega$ at a  boundary point $b$.  We call these submanifolds $\pi(b)$ {\it supporting hyperplanes} 
of $\Omega$.  The set $\mathcal P$ indexes the set of all supporting hyperplanes of $\Omega$, and the set ${\mathcal P}(b) \subset {\mathcal P}$ denotes the set of supporting hyperplanes 
at the boundary point $b$.   

\begin{remark}\label{sph.cvx} 

We  are considering the sphere of radius one, that is, the space of constant curvature $+1$. Otherwise, if the curvature is different from one, a constant factor has to be inserted in the trigonometric formulae. Then for a point $x$ in the convex set 
$\Omega$ and a hyperplane (great sphere) $\pi(b)$ supporting $\Omega$ at a boundary point $b \in \partial \Omega$, note that $d(x, \pi(b))$ is at most $\pi/2$. We also note that unless the point $x$ is the center of the hemisphere bounded by $\pi(b)$, the nearest point projection of $x$ to $\pi(b)$ is single valued.    
 
\end{remark}

Now we define the Funk metric in spherical geometry. We use the angular metric on the sphere, and we shall assume that the diameter of the open convex set $\Omega$ is $<\pi/2$ for reasons that will become apparent immediately after the next definition.

\begin{definition}\label{d:Funk-s}
For a pair of points $x$ and $y$ in $\Omega \subset S^n$, we define the Funk (asymmetric) metric by 
\begin{equation*}
F(x, y)= 
\begin{cases}  \displaystyle  \log \frac{\sin d(x, b(x, y))}{\sin d(y, b(x,y))}& \text{if $x\not=y$,}
\\
0 &\text{if $x=y$}
\end{cases}
\end{equation*}
where (as in the Euclidean and the hyperbolic cases) the point $b(x, y)$ is the first intersection point of the boundary $\partial \Omega$ with the geodesic ray $\{ \exp_x (t \xi_{xy} \,\, | \,\, t>0 \}$ from $x$ through $y$, and $\xi_{xy}$ is the unit tangent vector at $x$ of the arc-length parameterized geodesic connecting $x$ and $y$. 
\end{definition}

Note that the sine function is strictly increasing for angles between $0$ and $\pi/2$, and this makes the Funk distance $F(x,y)$ always nonnegative.

We recall that for a given spherical right triangle $\triangle(A,B,C)$
with angles $\alpha, \beta$ and $\gamma = \pi / 2$ and with side lengths $a,b$ and $c$  
opposite to the vertices $A, B$ and $C$ respectively, we have the formula
\[
\sin b = \sin c \sin \beta.  
\]
In analogy with the hyperbolic case, we note that choosing a point $A'$ on the side  $c$ and letting  $C'$ be its nearest point projection on the side $a$, we have another right triangle $\triangle(A',B,C')$ with angles
$\alpha', \beta,\gamma' $, with $\gamma'= \pi/2$, and with side lengths $a',b,c'$  
opposite to the vertices $A', B$ and $C'$ respectively, satisfying 
\[
\sin b' = \sin c' \sin \beta.  
\]
Again, as the ratios $\sin b / \sin c$ and $ \sin b' / \sin c' $ are equal to $\sin \beta$, we regard the two triangles $\triangle(A,B,C)$ and $\triangle(A',B,C')$ as {\it similar} with side lengths being weighted by the function $\sin$.  

Following the same argument as in hyperbolic geometry, we have the following variational formula for the Funk metric:

\begin{theorem} The Funk metric (Definition \ref{d:Funk-s}) on a convex subset $\Omega \subset S^n$ is also given by:
\[
F(x, y) = \sup_{\pi \in {\mathcal{P}}} \log \frac{\sin d(x, \pi)}{\sin d(y,  \pi)}.
\]    
\end{theorem}

For the sphere,  we identify the value of the Minkowski functional for the Funk metric $F$ as
\begin{equation}\label{fos:1}
p_{\Omega, x}(\xi) = \sup_{\pi \in {\mathcal{P}}} \frac{\cos d(x,  T(x, \xi, \pi) ) }{\sin d(x,  T(x, \xi, \pi))} 
\| \xi \|,
\end{equation}
or equivalently as  
\begin{equation}\label{fos:2}
p_{\Omega, x}(\xi) = \sup_{\pi \in {\mathcal{P}}} \frac{\cos d(x, \pi) }{\sin d(x, \pi)} \langle \eta_{\pi}(x), \xi \rangle
\end{equation}
where $\eta_\pi$ is the unit tangent vector at $x$ whose 
direction is opposite to the one of the gradient vector of the functional $d(*, 
\pi)$ at $x$.  
These expressions appear naturally by following the same argument we have seen for the hyperbolic Funk metric.   We then have the following analogous statement for the spherical Funk metric, the proof of which is also almost identical to the hyperbolic case;

\begin{theorem}
The Finsler distance $d(x, y)$ induced by the Minkowski functional $p_{\Omega, x}$ (\ref{fos:1}) or (\ref{fos:2}) coincides with the spherical Funk metric $F$ on $\Omega
\subset S^n$.
\end{theorem}

\subsection{Convexity of $F(x, y)$ in the $y$-variable} 

Consider a point $x$ in $\Omega$ and a geodesic segment $\alpha: [0,1] \rightarrow 
S^n$ with $\alpha(0)=y$ where $y$ is any point in $\Omega$.  Then for a fixed great sphere $\pi$, and denoting bf by $\eta_\pi$ the vector field that is minus the gradient vector field of the functional $d(., 
\pi)$,   we have:
\[
\frac{d}{dt} \log \frac{\sin d(x, \pi)}{\sin d(\alpha(t), \pi)} \ = 
\frac{\cos d(\alpha(t), \pi)}{\sin d(\alpha(t), \pi)}
 \langle \dot{\alpha}(t), \eta_\pi (\alpha(t)) \rangle
\]
 and 
\begin{eqnarray*}
\frac{d^2}{dt^2} \log \frac{\sin d(x, \pi)}{\sin d(\alpha(t), \pi)} \Big|_{t=0} 
 & = &\frac{1}{\sin^2 d(y, \pi)} \langle \dot{\alpha}(0), \eta_\pi(x) \rangle^2 \\
 & &
+ \frac{\cos d(y, \pi)}{\sin d(y, \pi)} \Big( \frac{d}{dt} \langle \dot{\alpha}(t), \eta_\pi (\alpha(t)) \rangle \Big|_{t=0} \Big). 
\end{eqnarray*}
First note that for any $p \in \Omega$ and any supporting hyperplane $\pi \in {\mathcal P}$ we have $0 < d(p, \pi) < \pi/2$ as explained in Remark \ref{sph.cvx}. Hence the values of $\sin d(y, \pi)$ and $\cos d(y, \pi)$ are strictly positive for all $y \in \Omega$ and $\pi \in {\mathcal P}$.

Secondly note that the quantity $\langle \dot{\alpha}(t), \eta_\pi (\alpha(t)) \rangle $ is equal to 
$\| \dot{\alpha}(t) \| \cos \theta(t)$ where $\theta(t)$ is the angle between the velocity vector 
$\dot{\alpha}(t)$ and the unit vector $\eta_\pi(\alpha(t))$, and that the speed $\| \dot{\alpha}(t) \| $ is constant over time.  Now the positive curvature of $S^n$ implies that the angle $\theta(t)$ is decreasing in $t$, which in turns implies  
\[
\frac{d}{dt} \langle \dot{\alpha}(t), \eta_\pi (\alpha(t)) \rangle  >0.
\]   
Therefore the funtion $\displaystyle  \log \frac{\sin d(x, \pi)}{\sin 
d(\alpha(t), \pi)}$ is convex in $t$. 

Recall that the supremum of a set of convex functions is convex. Hence as a 
consequence of the variational formulation
\[
F(x, y) = \sup_{\pi \in {\mathcal{P}}} \log \frac{\sin d(x, \pi)}{\sin d(y,  \pi)}.
\]   
We have
\begin{theorem}\label{c:convex}
The Funk metric $F(x, y)$ defined on a convex set $\Omega \subset S^n$ is convex 
in the $y$-variable.   
\end{theorem}

Note that the result of Theorem \ref{c:convex} contrasts with the cases of hyperbolic geometry (Proposition \ref{prop:nonconvex}). As we already recalled, in Euclidean geometry, the Funk metric $F(x, y)$ is convex in $y$ (\cite{Y3}).

\section{Hilbert metrics and their projective geometry}
 
The symmetrization  of the hyperbolic and spherical Funk metrics by taking the arithmetic means provides a new set of Hilbert-type metrics, which we call the hyperbolic and spherical Hilbert metrics  respectively.  It is well known that the geometry of Hilbert metrics defined on convex sets in ${\mathbb R}^n$ is very much related to the projective geometry of ${\mathbb R}^{n+1}$. We will show below that the geometry of the hyperbolic/spherical Hilbert metrics, defined on convex sets of ${\mathbb H}^n$ and $S^n$ respectively,  are also described in terms of the projective geometry of ${\mathbb R}^{n+1}$.

  \subsection{The Hilbert metric in $\mathbb{H}^n$}
Let $\Omega$ be an open convex (possibly unbounded) subset of $\mathbb{H}^n$. 

We symmetrize the Funk metric by taking the arithmetic mean:
\begin{eqnarray*}
H(x, y) &= & \frac12 \Big( F(x, y) + F(y,x) \Big) \\ &=& \frac12  \log \Big[ \frac{\sinh d(x, b(x, y))}{\sinh d(y, b(x, y))} \cdot
\frac{\sinh d(y, b(y, x))}{\sinh d(x, b(y, x))} \Big]
\end{eqnarray*}
\begin{definition}
The metric $H(x,y)$ is a Hilbert-type metric on $\Omega$, and we call it the \emph{hyperbolic Hilbert metric} of $\Omega$. 
\end{definition}

For a convex set  $\Omega$ in ${\bf H}^d$, the geodesic  segment 
connecting $x$ and $y$ in $\Omega$ is a Funk geodesic realizing both lengths $F(x, y)$ 
and $F(y,x)$.  This implies that the geodesic segment is a Hilbert geodesic. 

The Hilbert metric satisfies the triangle inequality, obtained by adding both sides of the following pair of inequalities:
\[
F(x, y) + F(y,z) \geq F(x,z)
\]
and 
\[
F(z, y) + F(y,x)  \geq F(z, x)
\]

For any two points $x$ and $y$ in $X$, since the ength of the hyperbolic geodesic segment joining them realizes both distances $F(x,y)$ and $F(y,x)$, it is is also a Hilbert geodesic. The criterion for uniquenss of geodesics joining two points is the same as the one in the Euclidean case. More precisely, we have:

  \begin{proposition}[Hilbert geodesics] \label{prop:Hilbert-g}
 For any convex subset $\Omega$ of $\mathbb{H}^n$, the following holds:
 \begin{enumerate}
 \item the hyperbolic geodesic segments are Hilbert geodesics;
 \item the hyperbolic geodesics are the unique Hilbert geodesics joining their endpoints if and only if  there does not exist in $\partial \Omega$ two hyperbolic geodesic segments of nonempty interior which span a 2-dimensional totally geodesic subspace.
 \end{enumerate}
\end{proposition}
 
\subsection{The Hilbert metric in $S^n$}

We now consider the spherical counterpart. We define the spherical  Hilbert metric $H(x,y)$ as an arithmetic symmetrization of the  spherical Funk metric. We obtain a formula which is analogous to the  formula of the hyperbolic Hilbert metric, except that in
the formula for the spherical case, one repplaces the sinh function by the sine function. 
We assume that the convex set $\Omega$ is contained in an open hemisphere. Unlike the case of the Funk spherical metric, we do not need any more that $\Omega$ be contained in a sphere of diameter $<\pi_2$, for the value $H(x,y)$ is always nonnegative.

\begin{definition}[The spherical Hilbert metric]
For $x$ and $y$ in $\Omega$, the spherical Hilbert metric is defined by the formula
\begin{eqnarray*} 
H(x, y) &= & \frac12 \Big( F(x, y) + F(y,x) \Big) \\ &=& \frac12  \log \Big( \frac{\sin d(x, b(x, y))}{\sin d(y, b(x, y))} \cdot
\frac{\sin d(y, b(y, x))}{\sin d(x, b(y, x))} \Big)
\end{eqnarray*}
\end{definition}

Proposition \ref{prop:Hilbert-g} also holds in the spherical case, with spherical geodesics replacing hyperbolic geodesics in the statement.

\subsection{Cross ratio on $S^n$ and ${\mathbb H}^n$ \label{crossratio} }

Having introduced the Hilbert metrics on ${\mathbb H}^n$ and $S^n$, we note  that the quantities inside the logarithm are called  cross ratio and that they encode a projective geometric information, which relates the three geometries of ${\mathbb R}^n$, $S^n$ and $\mathbb{H}^n$.  We shall explain this in detail.  We remark here that even though one can define {\it signed} cross ratio by considering orientation of the geodesic segments,  in this article we will be concerned only with  {\it unsigned} cross ratio, meaning that its values are always positive. 

We start by recalling some classical facts. In the Euclidean plane ${\mathbb R}^2$, consider four  ordered distinct lines $l_1, l_2, l_3, l_4$ in the plane that are concurrent at a point $A$ and let $l$   be a  line that intersects these four lines at points $A_1, A_2, A_3, A_4$ respectively. Then it is well known that the cross ratio $[A_2, A_3, A_4, A_1]$ of the ordered quadruple $A_1, A_2, A_3, A_4$ does not depend on the choice of the line $l'$. We say that the cross ratio is a \emph{projectivity invariant}. Thus it follows that the projective transformations of $\mathbb{R}^n$,  leaving  a convex set $\Omega$ invariant,  are isometries of the Hilbert metric of $\Omega$.

As a matter of fact, Menelaus (Alexandria, 2nd century A.D.)  considered the above property not only on the Euclidean plane, but also on the sphere, where the lines  are the spherical geodesics, which are the great circles of the sphere.  Once there is a parallel between the Euclidean geometry and the spherical geometry, it is natural to expect to have the corresponding statement for the hyperbolic geometry.  We now define the cross ratio for the three geometries.  
\medskip

 \begin{definition}  Consider a geodesic line in Euclidean, hyperbolic and spherical geometry respectively, and let 
  $A_1, A_2, A_3, A_4$ be four ordered pairwise distinct points on that line.
 We define the cross ratio $[A_1, A_2, A_3, A_4]$, 
 in the Euclidean case, by:
 \[
 [A_2, A_3, A_4, A_1]_e :=\frac{A_2 A_4}{A_3 A_4}\cdot  \frac{A_3 A_1}{A_2 A_1},
 \]
 in the hyperbolic case, by:
 \[
[A_2, A_3, A_4, A_1]_h := \frac{\sinh A_2 A_4}{\sinh A_3 A_4} \cdot \frac{\sinh A_3 A_1}{\sinh A_2 A_1},
 \]
and in the spherical case, by:
 \[
[A_2, A_3, A_4, A_1]_s := \frac{\sin A_2 A_4}{\sin A_3 A_4}\cdot  \frac{\sin A_3 A_1}{\sin A_2 A_1},
 \]
where $A_i A_j$ stands for the distance between the pair of points $A_i$ and $A_j$, which is equal to the length of the line segment joining them.  (For this, we shall assume that in the case of spherical geometry the four points lie on a hemisphere;  instead, we could work in the elliptic space, that is, the quotient of the sphere by its canonical involution.)
 \end{definition}

In the Appendix, we will present a proof of the fact that  the quantity 
\[
 \frac{\sinh d(x, b(x, y))}{\sinh d(y, b(x, y))} \cdot 
\frac{\sinh d(y, b(y, x))}{\sinh d(x, b(y, x))}
\]
is a projective invariant in ${\mathbb H}^n$.  However, we now give another proof of the projective invariance using the geometry of the ambient space ${\mathbb R}^{n+1}$ for the projective model.

\subsection{Perspectivity and Hilbert metric isometries}\label{perspectivity}

We denote by $U^n$ the open upper hemisphere of $S^n$ equipped with the induced metric. Let $X$ and $X'$ belong to the set $\{\mathbb{R}^n, \mathbb{H}^n, U^n\}$.  We now define a category of maps, which provides isometries with respect to the Hilbert geometries.

\begin{definition} A map $P: X\to X'$ is a \emph{perspectivity}, or a \emph{perspective-preserving transformation} if it preserves geodesics and if it preserves the cross ratio of quadruples of points on geodesics. 
\end{definition}

(We note that these are classical terms, see e.g. Hadamard \cite{Had} or Busemann 
\cite{Busemann1953}. We also note that such maps arise indeed in perspective drawing.) The 
obvious examples of perspectivities are the projective transformations of ${\mathbb R}^n$ to itself.
In what follows, using well-known projective models in $\mathbb{R}^{n+1}$ of hyperbolic space $\mathbb{H}^n$ and of the sphere $S^n$, we define natural homeomorphisms between $\mathbb{R}^n$, $\mathbb{H}^n$ and the open upper hemisphere of $S^n$ which are perspective-preserving transformations. The proofs are elementary and  are based on first principles of geometry.  

In \cite{PY2}, we showed that these three cross-ratios are the manifestation of the same entity; they are obtained from each other via projection maps between familiar representatives of the three geometries in $\mathbb{R}^{n+1}$. More specifically, the sphere $S^n$ is the set of unit vectors in ${\mathbb R}^{n+1}$ with respect to the Euclidean norm 
\[
\|x\|_e^2:=x_1^2 + \cdots + x_n^2 + x_{n+1}^2 = 1
\]
and the hyperbolic space ${\mathbb H}^n$ is the set of ``vectors of imaginary norm $i$" with $x_{n+1}>0$ in  ${\mathbb R}^{n+1}$ with respect to the Minkowski norm 
\[
\|x\|_m^2:= x_1^2 + \cdots  + x_n^2 - x_{n+1}^2 = -1
\]
These models of the two constant curvature spaces are called ``projective" for the geodesics in the curved spaces are realized as the intersection of the unit sphere with the two-dimensional subspace of ${\mathbb R}^{n+1}$ through the origin of this space.  

 Let $P_s$ be the  projection map from the origin of $ {\mathbb R}^{n+1}$ sending  the hyperplane $\{ x_{n+1} = 1 \} \subset {\mathbb R}^{n+1}$ onto the open upper hemisphere $U^n$ of $S^n$.
 
 Let  $P_h$ be the  projection map  from the origin of $ {\mathbb R}^{n+1}$ of the unit disc of the hyperplane $\{ x_{n+1} = 1 \} \subset {\mathbb R}^{n+1}$ onto the hyperboloid ${\mathbb H}^n \subset {\mathbb R}^{n+1}$.  
 
In \cite{PY2}, we proved the following:

\begin{theorem}[Spherical Case] \label{th:s}
The map $P_s$ is a perspectivity.  In particular, the projection map $P_s$ preserves the values of the cross ratio; namely for a set of  four ordered pairwise distinct points
  $A_1, A_2, A_3, A_4$ aligned on a great circle in the upper hemisphere, 
  we have \[ [P_s(A_2),P_2(A_3),P_s(A_4),P_s(A_1)]_s=[A_2, A_3, A_4, A_1]_e.\]
\end{theorem}

As the proof is short and elementary, we will include it. 

\begin{proof} 
Let $u, v$ be the two points on the hyperplane $\{x_{n+1} =1\}$ and $P_s(u)=:[u], P_s(v)=:[v]$ be the points in $U$, and $d([u], [v])$ be the spherical distance between them. Let $\|x\|$ be the Euclidean norm of the vector $x \in {\mathbb R}^{n+1}$. We show that 
\[
\sin d( [u], [v] ) = \frac{\|u-v\|}{\|u\|\|v\|}.
\]
This follows from the following trigonometric relations:
\begin{eqnarray*}
\sin d( [u], [v] ) & = & \sin \Big[ \cos^{-1} \Big( \frac{u}{\|u\|} \cdot  \frac{v}{\|v\|} \Big)\Big]  \\
 & = &  \sqrt{1 - \cos^2 \Big[ \cos^{-1} \Big( \frac{u}{\|u\|} \cdot  \frac{v}{\|v\|} \Big) \Big]} \\
 & = & \sqrt{1- ( \frac{u}{\|u\|} \cdot  \frac{v}{\|v\|} \Big)^2}   
 \\
 & = &  \frac{1}{\|u\| \|v\|} \sqrt{\|u\|^2 \|v\|^2 - (u \cdot v)^2}  \\
 & = &  \frac{1}{\|u\| \|v\|} \times (\mbox{area of parallelogram spanned by $u$ and $v$}) \\
 & = & \frac{\|u-v\|}{\|u\| \|v\|}.
\end{eqnarray*} 

Now  for a set of  four ordered pairwise distinct points
  $A_1, A_2, A_3, A_4$ aligned on a great circle in the upper hemisphere, their spherical cross ratio $[A_2, A_3, A_4, A_1]_e$ is equal to the Euclidean cross ratio $[P_s(A_2), P_s(A_3), P_s(A_4), P_s(A_1)]_s$;
\[
\frac{ \sin d( [A_2], [A_4] )}{\sin d([A_3], [A_4])} \cdot \frac{ \sin d( [A_3], [A_1] )}{\sin d([A_2], [A_1])} 
= \frac{\frac{\|A_2-A_4\|}{\|A_2\| \|A_4\|}}{ \frac{\|A_3-A_4\|}{\|A_3\| \|A_4\|} } \cdot
\frac{\frac{\|A_3-A_1\|}{\|A_3\| \|A_1\|}}{ \frac{\|A_2-A_1\|}{\|A_2\| \|A_1\|} }  
= \frac{\|A_2-A_4\|}{\|A_3 - A_4\|} \cdot \frac{\|A_3-A_1\|}{\|A_2 - A_1\|}
\]

\end{proof}

\begin{theorem}[Hyperbolic Case] \label{th:h}
The map $P_h$ is a perspectivity. In particular, the projection map $P_h$ preserves cross ratios, namely for a set of  four ordered pairwise distinct points
  $A_1, A_2, A_3, A_4$ aligned on a geodesic in the upper hyperbolid, we have
  \[ [P_h(A_2),P_h(A_3),P_h(A_4),P_h(A_1)]_h=[A_2, A_3, A_4, A_1]_e.\]  
\end{theorem}

We quote the proof in \cite{PY2}. 
\begin{proof}
We follow the spherical case, where the sphere of the unit radius in ${\mathbb R}^{n+1}$  is replaced by the upper sheet of the sphere of radius $i$, namely the hyperboloid in ${\mathbb R}^{n, 1}$.  Let $u, v$ be the two points on the hyperplane $\{x_{0} =1\}$.  and $P_h(u)=:[u], P_h(v)=:[v]$ be the points in the hyperboloid, or, equivalently, the time-like vectors of unit (Minkowski) norm.   Denote by $d([u], [v])$ the hyperbolic length between the points and let $\|x\|$ be the Minkowski norm of the vector $x \in {\mathbb R}^{n,1}$.  We will show that 
\begin{eqnarray*}
\sinh d([u], [v]) = - \frac{\|u-v \|}{\|u\| \|v\|}. 
\end{eqnarray*}
Note that the number on the right hand side is positive, for $\|u\|, \|v\|$ are positive imaginary numbers, and $u-v$ is a purely space-like vector, on which the Minkowski norm of ${\mathbb R}^{n, 1}$ and the Euclidean norm of ${\mathbb R}^n$ coincide.  

This follows from the following trigonometric relations
\begin{eqnarray*}
\sinh d( [u], [v] ) & = & \sinh \Big[ \cosh^{-1} \Big( \frac{u}{\|u\|} \cdot  \frac{v}{\|v\|} \Big)\Big] 
\\
 & = &  \sqrt{\cosh^2 \Big[ \cosh^{-1} \Big( \frac{u}{\|u\|} \cdot  \frac{v}{\|v\|} \Big) \Big] -1 } \\
 & = & \sqrt{\Big( \frac{u}{\|u\|} \cdot  \frac{v}{\|v\|} \Big)^2 -1 }  
\\
 & = &  \frac{\sqrt{-1}}{\|u\| \|v\|} \sqrt{\|u\|^2 \|v\|^2 - (u \cdot v)^2 }   \\
 & = &  \frac{ \sqrt{-1}}{\|u\| \|v\|} \sqrt{-1}\times (\mbox{area of parallelogram spanned by $u$ and $v$}) \\
 & = & - \frac{\|u-v\|}{\|u\| \|v\|}.
\end{eqnarray*} 
The formula \[\sqrt{\|u\|^2 \|v\|^2 - (u \cdot v)^2 }  =  \sqrt{-1}\times (\hbox{area of parallelogram spanned by } u \hbox{ and } v) \]
 is as stated in Thurston's notes (Section 2.6 \cite{Th}). Now by the same argument as in the spherical case, the hyperbolic cross ratio $[A_2, A_3, A_4, A_1]_e $ is equal to the Euclidean cross ratio $[P(A_2),P(A_3),P(A_4),P(A_1)]_h$.
\end{proof}

In the second paragraph of \S \ref{crossratio}, we have already referred to the notion of projectivity invariance for the cross ratio in Euclidean space. We extend this notion tp the cross ratio in the spherical and hyperbolic spaces, by using the same definition. In the case of the sphere, we  restrict to configurations where all the points considered are contained in an open hemisphere. From this we have easily the following classical result, going back to Menelaus in the spherical case:
\begin{corollary}\label{cr.proj}
The spherical and hyperbolic cross ratios are projectivity invariants.
\end{corollary}

 This follows from the facts that the projection map $P_s$ and $P_h$ are both perspective-preserving transformations, and that cross ratio in the Euclidean space is  a projectivity invariant. 
 
 By translating these results in the language of Hilbert geometry, we have:

\begin{corollary}[Spherical Case]\label{co:s}
Consider an open convex set $\Omega$ in $S^n \subset {\mathbb R}^{n+1}$ contained in the upper hemisphere $\{\|x\|_e^2 =1, x_{n+1}>0\}$, and let $H$ be the spherical Hilbert metric of $\Omega$.  Let $\tilde{\Omega}$ be the image of  projection $P$ of $\Omega$ from the origin onto the hyperplane $\{ x_{n+1} = 1 \}$, an open convex set in the plane with its Euclidean  Hilbert metric $\tilde{H}$.  Then the projection map $P: (\Omega, H) \rightarrow (\tilde{\Omega}, \tilde{H})$ is an isometry.   
\end{corollary}

\begin{corollary}[Hyperbolic Case]\label{co:h}
Consider an open convex set $\Omega$ in ${\mathbb H}^n \subset {\mathbb R}^{n, 1}$ where the inclusion is the isometric embedding of ${\mathbb H}^n$ as the hyperboloid $\{ \|x\|_m^2 = -1 \}, x_{0} >0\}$, and let $H$ be the hyperbolic Hilbert metric of $\Omega$.  Let $\tilde{\Omega}$ be the image of  projection $P$ of $\Omega$ from the origin onto the hyperplane $\{ x_{0} = 1 \}$, an open convex set in the plane with its Euclidean  Hilbert metric $\tilde{H}$.  Then the projection map $P: (\Omega, H) \rightarrow (\tilde{\Omega}, \tilde{H})$ is an isometry.   
\end{corollary}

\subsection{On Hilbert's Problem IV}
We already mentioned Hilbert's Fourth Problem. Since the projection maps defined in \S \ref{perspectivity} between the  the Euclidean plane $\mathbb{R}^n$, the upper hemisphere $U^n\subset S^n$ and the hyperbolic plane $\mathbb{H}^n$ preserve convexity and send lines to lines, Hilbert's problem IV in the three spaces of constant curvature amount to one and the same problem.

\subsection{Projective transformations as isometries of Hilbert metrics}

We remark that by Corollaries \ref{co:s} and \ref{co:h} , which state that the projection maps $P_h$ and $P_s$ are perspectivities, each convex set $\Omega \subset \mathbb{H}^n$ corresponds to a convex set $P_h^{-1}(\Omega) \subset \mathbb {R}^n$ with the hyperbolic Hilbert metric $H_\Omega$ isometric to  the Euclidean Hilbert metric $H_{P_h^{-1}(\Omega)}$, and that each convex set $\Omega' \subset U^n \subset S^n$ corresponds to a convex set $P_s^{-1}(\Omega')$  with $H_{\Omega'}$ isometric to $H_{P_s^{-1}(\Omega')}$.   We also  note that although the set $P_h^{-1}(\Omega)$ is always properly contained in the unit ball in ${\mathbb R}^n$ by construction,  by recalling that Euclidean homotheties are isometries of Hilbert metrics,  the size of   $P_h^{-1}(\Omega)$ is irrelevant in characterizing the Hilbert metrics. 

As a simple application of this phenomenon, we note the following example of Hilbert geometry: 

\subsubsection{Hilbert metrics as models of the hyperbolic plane}
 Denote by  $\mathbb{D}^n_R \subset  \{x_{n+1} = 1\}$ a disk of radius $R>0$ centred at the origin of the Euclidean plane. By restricting the inverse map of  $P_s$  to the disk $\mathbb{D}_R$, we obtain a map from this disk into the sphere $S^n$.  For $R < 1$, by restricting the inverse map of $P_h$ to $\mathbb{D}_R$, we obtain a map from this disk into the hyperboloid/hyperbolic space $\mathbb{H}^n$.  These maps are  isometries for the Hilbert metrics as seen in Corollaries \ref{co:s} and \ref{co:h}. Since the disk $\mathbb{D}^n_R$ equipped with its Hilbert metric is a model of hyperbolic geometry (the Beltrami-Klein model), we obtain in this way new models of hyperbolic space which sit  in hyperbolic space and in the sphere respectively. We call such models \emph{generalized Beltrami-Klein models}, since they are defined using the spherical and the hyperbolic cross ratios respctively. Thus we identify, in each of the positively/negatively curved spaces, a nested family of geodesic balls as models of the hyperbolic space. Note that the limit of these models in $S^n$ and in $\mathbb{H}^n$ as the radius $R$ goes to $\pi/2$ and $\infty$ respectively, is the upper cap $U^n$ and the entire space $\mathbb{H}^n$ whose Hilbert metrics, if we define them by the formula we used for proper open convex subsets, are identically zero. \\

 An immediate consequence of the fact that  the generalized cross ratio is a projective invariant is that a projective transformation $\phi$ of ${\mathbb R}^n$, ${\mathbb H}^n$ or $S^n$ (see Corollary
 \ref{cr.proj} and Appendix) induces an isometry of the Hilbert metric $H_\Omega$ of a convex set $\Omega$, namely the following two spaces are isometric via $\phi$:
 \[
 (\Omega, H_\Omega) \cong^{\rm isom} (\phi( \Omega), H_{\phi (\Omega)})
 \]
 Here a projective transformation is a  diffeomorphism of the space form  ${\mathbb R}^n$, ${\mathbb H}^n$ or $S^n$, sending geodesics  to geodesics, induced by a linear map of the ambient space ${\mathbb R}^{n+1}$ of the projective models of the three spaces.   Every isometry of the space forms ${\mathbb R}^n$, ${\mathbb H}^n$ and $S^n$ is trivially a projective transformation.   

In ${\mathbb R}^n$, the set of linear  projective transformation is exactly the general linear group $\mathrm{GL}(n, {\mathbb R})$, a strictly larger set than the set of the linear isometries $\mathrm{O}(n)$.  

In hyperbolic geometry, however, each projective transformation is an isometry of the space. Indeed, a projective transformation of the hyperbolic space ${\mathbb H}^n$, sending each hyperbolic geodesic to a hyperbolic geodesic, is a map which sends a circle to a circle when the hyperbolic space is modeled on the open unit ball in ${\mathbb R}^n$. But such a map is known to be an isometry of ${\mathbb H}^n$.  Hence, 

\begin{theorem}\label{hyp.iso}
Given a convex set $\Omega$ in the hyperbolic space, a diffeomorphism $\phi$ is an isometry of the Hilbert metric, namely, $ (\Omega, H_\Omega) \cong^{\rm isom} (\phi( \Omega), H_{\phi (\Omega)})$ if and only if $\phi$ is a hyperbolic isometry.
\end{theorem}

In spherical geometry, on the other hand, each projective transformation of $S^n$ can be regarded as an element of the general linear group $\mathrm{GL}(n+1, {\mathbb R})$ as the transformation should send each great circle of $S^n$ to a great circle, which corresponds to a linear self map of ${\mathbb R}^{n+1}$, for  it sends each hyperplane containing the origin to another, where each of those hyperplanes corresponds to a great sphere in $S^n$.  Therefore we have the following; 

\begin{theorem} Given a convex set $\Omega$ in $S^n$, each element $\phi$ of $\mathrm{GL}(n+1, {\mathbb R})$ induces an isometry of the Hilbert metric $H_\Omega$;  $ (\Omega, H_\Omega) \cong^{\rm isom} (\phi( \Omega), H_{\phi (\Omega)})$.  
\end{theorem}

This observation in effect says that among the three geometries of $\mathbb{R}^n$, $\mathbb{H}^n$ and $S^n$, the Hilbert geometries of convex sets on each space are mutually identifiable. In light of this, we can state the following: 

\begin{corollary}
The isometries of a Hilbert metric defined on a convex set in one of the spaces $\mathbb{R}^n, S^n,\mathbb{H}^n$  are always  realized as those induced from the isometries of $\mathbb{H}^n$.  
\end{corollary}

  \section*{{\bf Appendix}\\ Classical proofs of some results in projective geometry \\ on $\mathbb{H}^n$ and $S^n$}\label{cross-ratio}
  
\subsection{Projective invariance of cross ratio}

  
  We recall the following well-known formulae for the purpose of highlighting the analogy (at the formal level) between the three geometries, and because we shall use these formulae below in the invariance of the cross ratio. 
 
  \begin{proposition}[Sine Rule]
  Given a triangle $ABC$ with sides $a,b,c$ opposite to the angles $A,B,C$ respectively, we have, in the case where the triangle is Euclidean:
\[\frac{a}{\sin A} =\frac{b}{\sin B}= \frac{c}{\sin C},\] 
in the case where the triangle is spherical:
\[\frac{\sin a}{\sin A} =\frac{\sin b}{\sin B}= \frac{\sin c}{\sin C},\] 
and in the case where the triangle is hyperbolic:
\[\frac{\sinh a}{\sin A} =\frac{\sinh b}{\sin B}= \frac{\sinh c}{\sin C}.\] 

\end{proposition}
For the proof, and for other trigonometric formulae in hyperbolic trigonometry we refer the reader to \cite{ACP} where the proofs are given in a model-free setting. In such a setting the proofs work as well in the spherical geometry case.

 We now present a few results in the setting of hyperbolic geometry. The proofs are based on the Sine Rule and they mimic the proofs of corresponding results in Euclidean geometry. From our point of view, the fact that is most interesting is the formal analogy and the results and the formulas in the three cases : Euclidean, hyperbolic and spherical.

\begin{proposition} \label{prop11} Let $ABC$ be a hyperbolic triangle. We join $A$ by a geodesic to a point $D$ on the line $BC$. Then we have 
\[\frac{\sinh DC}{\sinh BD}=\frac{\sin \widehat{DAC}}{\sin  \widehat{BAD}}\cdot \frac{\sin B}{\sin C}.\]
\end{proposition}

 \begin{figure}[!hbp]   
 \centering   
 \psfrag{B}{\small $B$} \psfrag{A}{\small $A$} 
 \psfrag{C}{\small $C$} \psfrag{D}{\small $D$} 
 \includegraphics[width=0.35\linewidth]{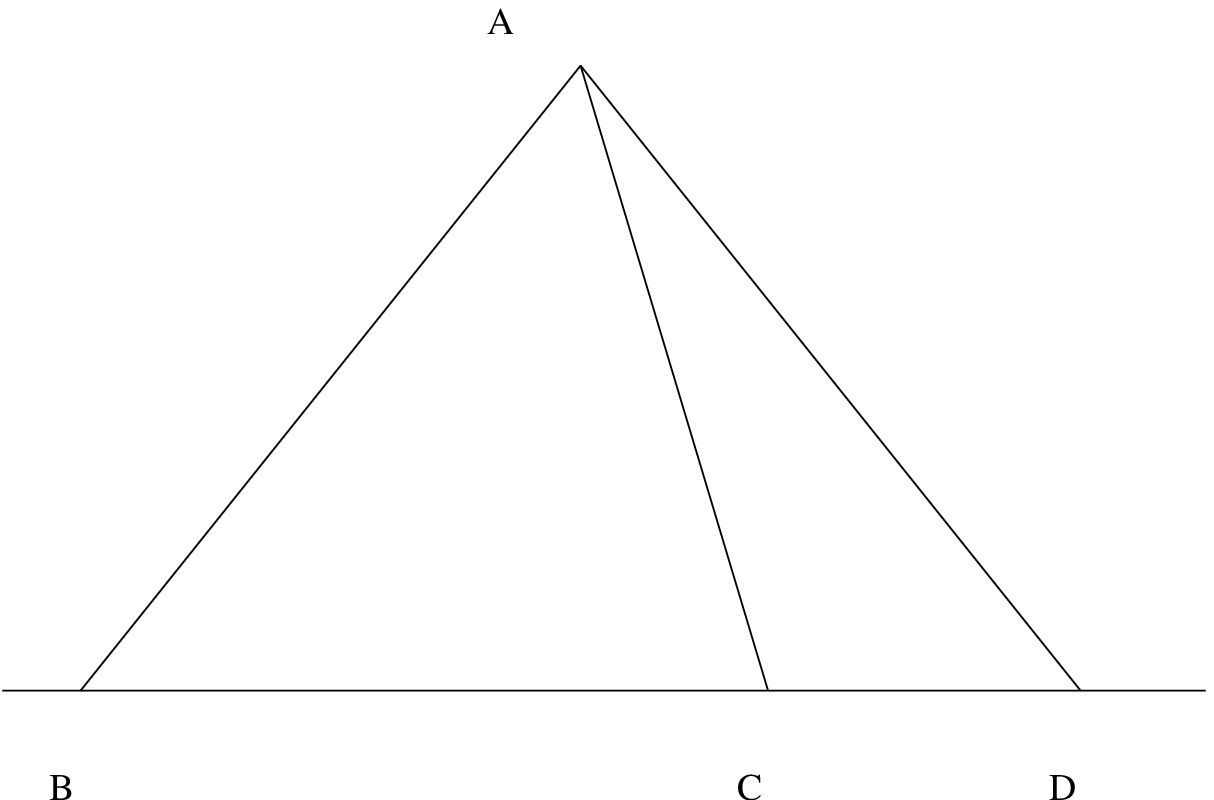}   
 \caption{\small {}}     
 \label{Funk2}    
 \end{figure}

\begin{proof} There proof is independent of whether the point $D$ is between $B$ and $C$, or $C$ is between $D$ and $B$, or $B$ is between $C$ and $D$. 

Applying the sine rule in the triangle $DAC$, we have
\[\frac{\sinh DC}{\sinh AD} = \frac{\sin \widehat{CAD}}{\sin C}.\]
Applying the sine rule in the triangle $BAD$, we have
\[\frac{\sinh AD}{\sinh BD} = \frac{\sin B}{\sin \widehat{BAD}}.\]
From the last two equations, we obtain 
\[\frac{\sinh DC}{\sinh BD} = \frac{\sin \widehat{DAC}}{\sin \widehat{BAD}}\cdot  \frac{\sin B} {\sin C}.\]
This proves Proposition \ref{prop11}.
\end{proof}

\begin{proposition} \label{prop2}
Consider four ordered distinct geodesic lines $l_1,l_2,l_3,l_4$ intersecting at a point $A$ and let $l$ be a geodesic line intersecting $l_1,l_2,l_3,l_4$ at points $A_1, A_2, A_3, A_4$ respectively. Then we have
\[\frac{\sinh A_2 A_4}{\sinh A_3A_4}\cdot \frac{\sinh A_3A_1}{\sinh A_2A_1}= \frac{\sin \widehat{A_2AA_4}}{\sin \widehat{A_3AA_4}}\cdot \frac{\sin \widehat{A_3AA_1}}{\sin \widehat{A_2AA_1}}.
\]
\end{proposition}

\begin{figure}[!hbp]   \centering     \psfrag{A}{\small $A_1$} \psfrag{C}{\small $A_3$} \psfrag{D}{\small $A_4$} \psfrag{B}{\small $A_2$} \psfrag{l}{\small $l_1$} \psfrag{m}{\small $l_2$} \psfrag{n}{\small $l_3$}  \psfrag{o}{\small $l_4$} \includegraphics[width=0.45\linewidth]{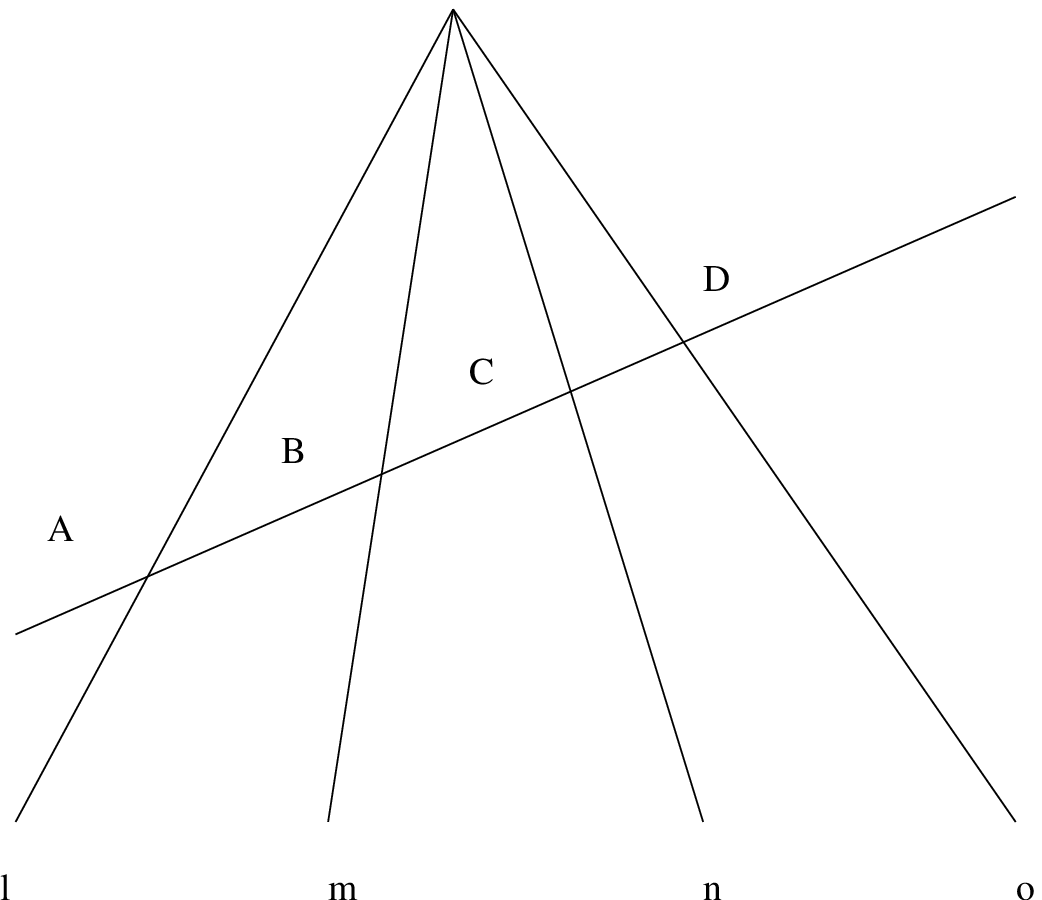}      \caption{\small {}}     \label{Funk3}       \end{figure}

\begin{proof}  We refer to Figure \ref{Funk3}.
Using Proposition \ref{prop11}, we have
\[\frac{\sinh A_2A_4}{\sinh A_3A_4}= \frac{\sin \widehat{A_2AA_4}}{\sin \widehat{A_3AA_4}}\cdot \frac{\sin A_3}{\sin A_2}.
\]
Using again Proposition \ref{prop11}, we have
\[
\frac{\sinh A_3A_1}{\sinh A_2A_1}= \frac{\sin \widehat{A_3AA_1}}{\sin \widehat{A_2AA_1}}\cdot \frac{\sin A_2}{\sin A_3}.
\]
This implies 
\[\frac{\sinh A_2A_4}{\sinh A_3A_4} \cdot \frac{\sinh A_3A_1}{\sinh A_2A_1} = \frac{\sin \widehat{A_2AA_4}}{\sin \widehat{A_3AA_4}} \cdot \frac{\sin \widehat{A_3AA_1}}{\sin \widehat{A_2AA_1}}.
\]
This proves Proposition \ref{prop2}.
\end{proof}

\begin{corollary}[Cross ratio invariance in hyperbolic geometry] \label{cor:cross-inv}
Consider four  ordered distinct geodesics $l_1, l_2, l_3, l_4$ in the hyperbolic plane that are concurrent at a point $A$ and let $l$   be a  geodesic that intersects these four lines at points $A_1, A_2, A_3, A_4$ respectively. Then the cross ratio of the ordered quadruple $A_1, A_2, A_3, A_4$ does not depend on the choice of the line $l'$. 
\end{corollary}

\begin{proof}
By Proposition \ref{prop2}, the cross ratio of the quadruple $A_1, A_2, A_3, A_4$ depends only on the angles that are made by the lines $l_1, l_2, l_3, l_4$.
\end{proof}
The result of Corollary \ref{cor:cross-inv} is a hyperbolic analogue of the result saying that the Eucidean cross ratio is a projective invariant.
We now present the theorem of Menelaus in hyperbolic  geometry which is used in section \ref{s:Funk-H}, with a proof based on the Sine Rule.\footnote{Menelaus (2nd century A.D.) in his \emph{Sphaerica} gave the theorem in the Euclidean and in the spherical cases. There are several proofs of that theorem, and the proof that we give is certainly not the original proof given by Menelaus, since he did not have the Sine Rule at his disposal. In fact, Menelaus did not formulate his theorem in terms of sines, but of chords.  We recall by the way the notion of sine was introduced by the Arabs in the ninth century, who also discovered the Sine Rule, in Euclidean and in spherical geometry. Let us also note that a modern formulation of the theorem involves algebraic ratios and not just distances, and in this setting the right hand side of the equality given in Proposition \ref{prop:M} is $-1$ instead of $1$.} 

\begin{figure}[!hbp] \centering \psfrag{B}{\small $B$} \psfrag{A}{\small $A$} \psfrag{C}{\small $C$} \psfrag{D}{\small $A'$} \psfrag{E}{\small $C'$} \psfrag{F}{\small $B'$} \includegraphics[width=0.65\linewidth]{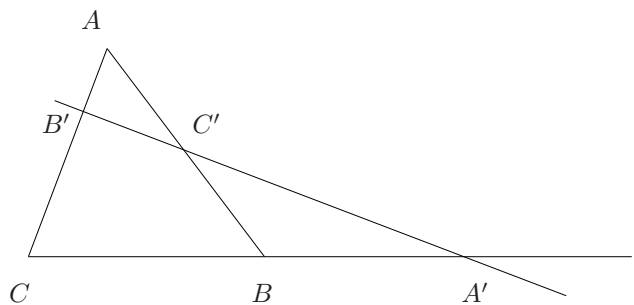}    \caption{\small {The theorem of Menelaus says that three points $A',B',C'$ on the lines containing the sides $BC, AC, AB$ of a triangle are aligned if and only they satisfy the relation  in Proposition \ref{prop:M}.}}   \label{Funk4}  \end{figure}

\begin{proposition}[Menelaus' Theorem in the hyperbolic plane] \label{prop:M} Let $ABC$ be a triangle in the hyperbolic plane ad let $A',B',C'$ be three points on the lines containing the sides $BC, AC, AB$. Then, the points $A',B',C'$ are aligned if and only if we have the relation
\[\frac{\sinh AC'}{\sinh AB'} \cdot \frac{\sinh BA'}{\sinh BC'}\cdot \frac{\sinh CB'}{\sinh A'C}= 1\]
\end{proposition}

\begin{proof} The proof is analogous to the proof of this theorem in the Euclidean case. We just consider prove here the ``only if" direction in order to show how one can use the hyperbolic sine rule instead of the Euclidean one. Needless to say, the proof works as well in the spherical case, with $\sinh$ replaced by $\sin$.

We consider the case of Figure \ref{Funk4}.

In the triangle $AB'C'$, the Sine Rule gives: 
\[
\frac{\sin \widehat{AB'C'}}{\sin \widehat{AC'B'}}=\frac{\sinh AC'}{\sinh AB'}.\]
In the triangle $BC'A'$, the Sine Rule gives: 
\[
\frac{\sin \widehat{BC'A'}}{\sin \widehat{BA'C'}}=\frac{\sinh BA'}{\sinh BC'}.\]
In the triangle $CA'B'$, the Sine Rule gives: 
\[
\frac{\sin \widehat{CA'B'}}{\sin \widehat{A'B'C}}=\frac{\sinh CB'}{\sinh A'C}.\]

Multiplying the three sides of the last three equations and using the fact that
$\sin \widehat{AB'C'}=\sin \widehat{A'B'C}$ and $\sin \widehat{BC'A'}=\sin \widehat{AC'B'}$, we obtain
\[1=\frac{\sinh AC'}{\sinh AB'} \cdot \frac{\sinh BA'}{\sinh BC'}\cdot \frac{\sinh CB'}{\sinh A'C},\]
which is the desired equality.

\end{proof}

The spherical case of all these results can be done in a similar way; the function $\sinh$ should be replaced in this case by the function $\sin$.

\subsection{An alternative proof of the triangle inequality for $F(x, y)$} 

The proof uses a drawing (Figure \ref{Funk5}) and Menelaus' Theorem (cf. Proposition \ref{prop:M}). We present it here, using the notation of \cite{Zaustinsky}.

Let $x,y,z$ be three points in $\Omega$. In the case where the three points are collinear (that is,  if they are on a common hyperbolic geodesic), we know that the triangle inequality is satisfied, and that it is an equality, that is, we have $F(x,y)+ F(y,z)= F(x,z)$ if $x,y,z$ lie in that order on the line. 
Assume now the three points are not collinear and let $a,b,c,d,e,f$ the intersections with $\partial \Omega$ of the lines $xz$, $yx$ and $zy$, with the orders of the various quadruples of collinear points represented in Figure \ref{Funk5}. 

\begin{figure}[!hbp]  \centering    \psfrag{B}{\small $b'$} \psfrag{A}{\small $a'$} \psfrag{a}{\small $a$} \psfrag{b}{\small $b$} \psfrag{c}{\small $c$} \psfrag{d}{\small $d$} \psfrag{e}{\small $e$} \psfrag{f}{\small $f$} \psfrag{p}{\small $p$} \psfrag{g}{\small $g$} \psfrag{x}{\small $x$} \psfrag{y}{\small $y$}\psfrag{z}{\small $z$} \includegraphics[width=0.35\linewidth]   {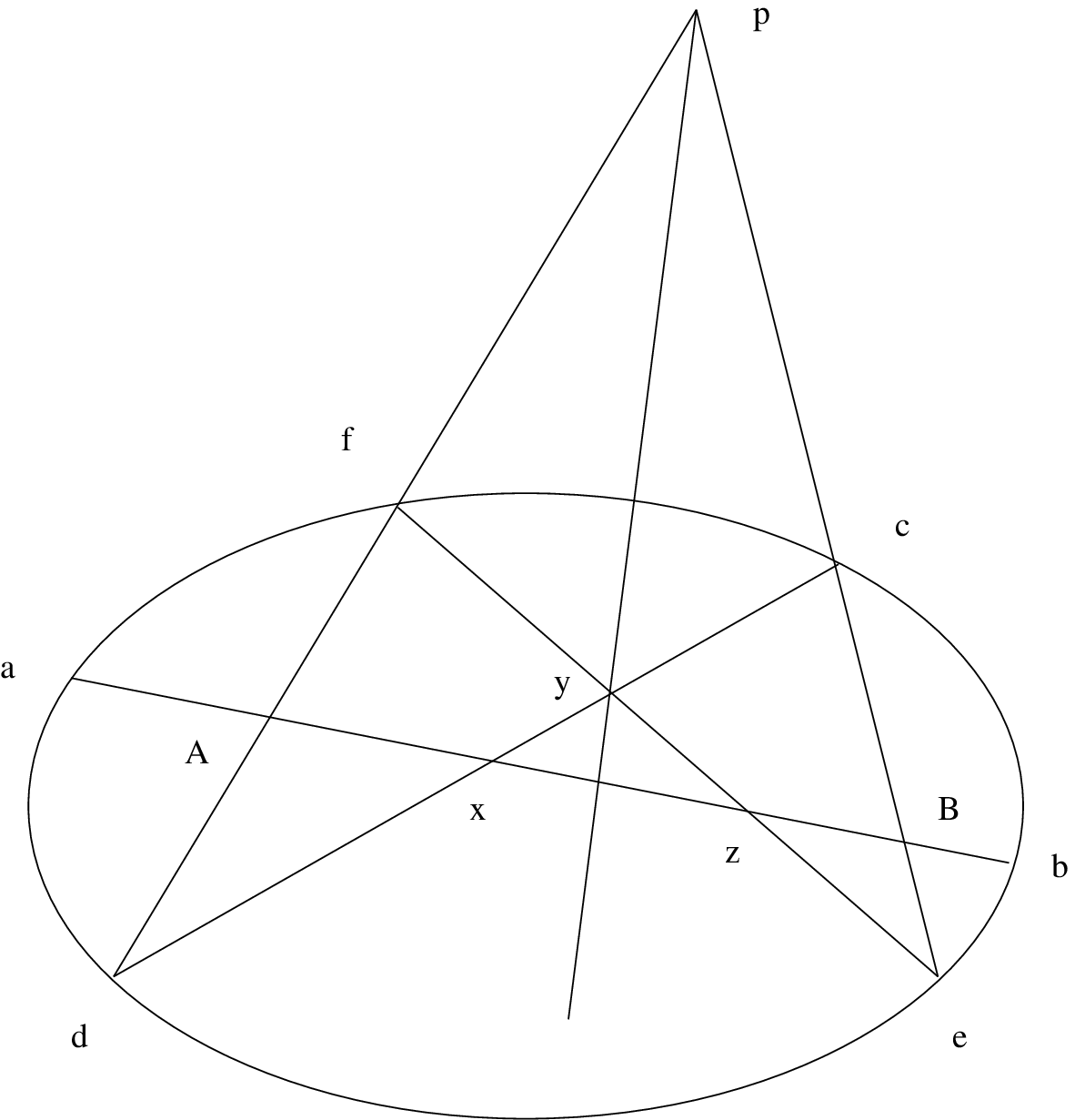} \caption{\small {}}     \label{Funk5}    \end{figure}

By the cross ratio invariance in hyperbolic geometry (Proposition \ref{cor:cross-inv}), we have 
\[\frac{\sinh xc}{\sinh yc} \cdot \frac{\sinh dy}{\sinh dx} =\frac{\sinh xb'}{\sinh gb'} \cdot \frac{\sinh a'g}{\sinh a'x}\]

and 

\[\frac{\sinh ye}{\sinh ze} \cdot \frac{\sinh fz}{\sinh fy} =\frac{\sinh gb'}{\sinh zb'} \cdot \frac{\sinh a'z}{\sinh a'g}.\]

Multiplying both sides of these two equations, we get

\[\frac{\sinh xc}{\sinh yc} \cdot \frac{\sinh ye}{\sinh ze} =\frac{\sinh xb'}{\sinh zb'} \cdot \frac{\sinh a'z}{\sinh a'x}\cdot \frac{\sinh dx}{\sinh dy} \cdot \frac{\sinh fy}{\sinh fz}.\]
By the Theorem of Menelaus (Theorem \ref{prop:M}) applied to the triangle $fa'z$, we have
\[\frac{\sinh dx}{\sinh dy} \cdot \frac{\sinh fy}{\sinh fz} =\frac{\sinh a'x}{\sinh a'z}.\]
This gives
\[\frac{\sinh xc}{\sinh yc} \cdot \frac{\sinh ye}{\sinh ze} =\frac{\sinh xb'}{\sinh zb'}\geq \frac{\sinh xb}{\sinh zb},\]
and the inequality is strict unless $b=b'$.
From this we see that the inequality is strict for all $x,y,z$ unless $\partial \Omega$ contains a hyperbolic segment. This gives the desired result.

\end{document}